\documentclass[12pt]{amsart}

\usepackage{amssymb, amsmath, xcolor}
\usepackage{import}
\usepackage{tikz}
\usetikzlibrary{shapes, patterns}
\tikzset{filled/.style={minimum width=5pt,inner sep=0pt,circle,fill=black}}
\usetikzlibrary{calc}
\usepackage{subcaption}
\usepackage{mathtools}
\usepackage[utf8]{inputenc}
\usepackage{float}
\usepackage{todonotes}
\usepackage[square,numbers]{natbib}
\usepackage{qtree}

\newtheorem{theorem}{Theorem}[section]
\newtheorem{lemma}[theorem]{Lemma}

\newtheorem{question}[theorem]{Question}

\theoremstyle{definition}
\newtheorem{definition}[theorem]{Definition}

\theoremstyle{remark}
\newtheorem{remark}[theorem]{Remark}

\numberwithin{equation}{section}
\numberwithin{figure}{section}

\renewcommand{\mod}{\operatorname{mod}}
\newcommand{\N}{\mathbb{N}}
\newcommand{\Z}{\mathbb{Z}}

\title[Chromatic Numbers Local Constraints]{Chromatic Numbers with Closed Local Modular Constraints}

\author[Herden, Meddaugh, Sepanski, $\ldots$]{Daniel Herden, Jonathan Meddaugh, Mark R. Sepanski, William Clark, Adam Kraus, Ellie Matter, Kyle Rosengartner, Elyssa Stephens, John Stephens, Mitchell Minyard, Kingsley Michael, Maricela Ramirez}

\thanks{The first author was supported by Simons Foundation grant MPS-TSM-00007788.
	The second author was supported by a grant from the Simons Foundation (960812, JM)}

\address{
All authors:
Department of Mathematics,
Baylor University,
Sid Richardson Building,
1410 S.~4th Street,
Waco, TX 76706, USA}

\email{daniel\_herden@baylor.edu,  jonathan\_meddaugh@baylor.edu, mark\_sepanski@baylor.edu, william\_clark2@baylor.edu, adam\_kraus1@baylor.edu, ellie\_carr3@baylor.edu, kyle\_rosengartner1@baylor.edu, ellie\_cirillo1@baylor.edu, john\_stephens2@baylor.edu, mitch\_minyard1@baylor.edu, kingsley\_michael1@baylor.edu, maricela\_ramirez1@baylor.edu}

\date{\today}

\begin{document}


\begin{abstract}
    Generalizing the notion of odd-sum colorings, a $\Z$-labeling of a graph $G$ is called a \emph{closed coloring with remainder $k\mod n$} if the closed neighborhood label sum of each vertex is congruent to $k\mod n$.
    If such colorings exist, we write $\chi_{n,k}(G)$ for the minimum number of colors used for a closed coloring with remainder $k\mod n$ such that no neighboring vertices have the same color. 
    General estimates for $\chi_{n,k}(G)$ are given along with evaluations of $\chi_{n,k}(G)$ for some finite and infinite order graphs. 
\end{abstract}

\maketitle

\tableofcontents

\section{Introduction}

The chromatic number $\chi(G)$, the minimum number of colors needed for a proper coloring of the graph $G$, is one of the most well studied invariants of graph theory and 
still an object of active research 
\cite{cambie2024removing, campena2023graph, cervantes2023chromatic, char2023improved, haxell2023large, heckel2021non, isaev2021chromatic, matsumoto2021chromatic}.
Two related concepts, known as \emph{odd colorings} and \emph{odd-sum colorings}, have also been well studied \cite{RemarksOdd, CaroYairPetru2023, CranstonOddSum, Cranston1Planar, KnauerBoundedness}. 
Here, a coloring is called an odd coloring if for every vertex $x$ there exists some color that appears an odd number of times among the vertices in its (open) neighborhood $N(x)$, while a 
$\Z$-labeling is called an odd-sum coloring if the sum of the labels over every closed neighborhood $N(x) \cup \{x\}$ is congruent to $1\mod 2$. In that case, $\chi_{\text{o}}(G)$ and $\chi_{\text{os}}(G)$ denote the minimum numbers of colors used for a proper odd coloring and proper odd-sum coloring, respectively.

In \cite{NeighborhoodParity}, Petru\v{s}evski
and \v{S}krekovski introduced the odd chromatic number $\chi_{\text{o}}(G)$ of a graph $G$ and found the general upper bound $\chi_{\text{o}}(G)\le 9$ for planar graphs. Petr and Portier \cite{PetrChiPlanar} tightened this bound to $\chi_{\text{o}}(G)\le 8$, and Cranston~\cite{CranstonSparse} further sharpened this bound based on the average degree of the graph~$G$. The odd-sum chromatic number $\chi_{\text{os}}(G)$ was introduced in \cite{CaroYairPetru2023} to obtain tight upper-bounds for planar, outerplanar, and bipartite graphs as well as various inequalities for general nonempty graphs, including $\chi_{\text{os}}(G)\le 2\chi(G)$.

In this work, we study a generalization of this concept. We say that a $\Z$-labeling of $G$ is a
\emph{closed coloring with remainder $k\mod n$} if the sum of the labels over the \emph{closed neighborhood} $N[x]=N(x)\cup \{x\}$ of each vertex $x$ is congruent to $k\mod n$. If such colorings exist, 
the \emph{closed chromatic number of $G$ with remainder $k \mod n$}, written
$\chi_{n,k}(G)$, is the minimum number of colors used for a proper closed coloring with remainder $k\mod n$. With this notation, $\chi_{\text{os}}(G) = \chi_{2,1}(G)$.

Definitions are given in Section~\ref{sec: defs}. Basic results and inequalities are given in Section~\ref{sec: basic results}. Finite order examples, including complete graphs, stars, friendship graphs, paths, complete bipartite graphs, regular graphs, and cycles are studied in Section~\ref{sec: finite order exs}. Infinite order examples, including the complete $m$-ary rooted tree and the regular tilings of the plane are studied in Section~\ref{sec: infinite order exs}. Certain finite trees are examined in Section~\ref{sec: trees}. In these cases, existence of $\chi_{n,k}(G)$ can be quite subtle, see Theorem \ref{thm: perfect binary trees} on rooted perfect binary trees. Finally, Section~\ref{sec: gen petersen} studies existence of $\chi_{n,k}(G)$ for generalized Petersen graphs.

\section{Definitions}\label{sec: defs}

We write $\N$ for the nonnegative integers and $\Z^+$ for the positive ones. 
For $a,b \in \Z$, not both zero, we write $(a,b)$ for the greatest common divisor of $a$ and $b$. 
For $k\in\Z$ and $n\in\Z^+$, we write $[k]$ for the image of $k$ in $\Z_n$.

We write $G=(V,E)$ for a simple graph with vertex set $V$ and edge set $E$ and \[\ell:V\longrightarrow \Z\] for a \emph{coloring} or \emph{labeling} of the vertices by $\Z$.
The \emph{order} of a labeling, $|\ell|$, is the size of its image. 

If $v\in V$, 
the \emph{open neighborhood} of $v$, $N(v)$, consists of all vertices adjacent to $v$ and 
the \emph{closed neighborhood} of $v$, $N[v]=N(v)\cup \{v\}$, consists of $v$ and all vertices adjacent to $v$. 
A labeling is called \emph{proper} if $\ell(v)\neq\ell(w)$ for each $v\in V$ and each $w\in N(v)$. The \emph{chromatic number} of $G$, $\chi(G)$, is the minimum order of a proper labeling of $G$.

In the following definition, recall that our labelings $\ell$ have codomain~$\Z$.
\begin{definition}\label{def: coloring with remainder k mod n}
    Let $k\in\Z$ and $n\in\Z^+$. 
    
    A \emph{closed coloring with remainder $k\mod n$} of $G$ is a labeling $\ell$ of $G$ so that, for each $v\in V$,
    \[ \sum_{w\in N[v]} \ell(w) \equiv k \mod n. \]
    If no proper closed coloring with remainder $k\mod n$ of $G$ exists, we say that $\chi_{n,k}(G)$ does not exist.
    Otherwise, if proper closed colorings with remainder $k\mod n$ of $G$ exist  of finite order, the \emph{closed chromatic number of $G$ with remainder $k \mod n$}, written \[\chi_{n,k}(G),\] is the minimum order of a proper closed coloring with remainder $k\mod n$ of $G$. If such colorings exist only of infinite order, we write \mbox{$\chi_{n,k}(G)=\infty$}. 
\end{definition}
Note that $\chi_{n,k}(G)$ only depends on $n$ and the residue class $k\mod n$.
Moreover, the case of $\chi_{2,1}(G)$ in Definition \ref{def: coloring with remainder k mod n} coincides with the notion of the \emph{odd-sum chromatic number} of $G$, $\chi_{\text{os}}(G)$, introduced in \cite{CaroYairPetru2023}. 

\section{Basic Results}\label{sec: basic results}

When $\chi_{n,k}(G)$ exists, we certainly have
\[ \chi(G) \leq \chi_{n,k}(G).\]
However, 
as seen from the following theorem, the case of $k=0$ does not provide a new invariant. 
\begin{theorem}
    Let $n\in\Z^+$. If $\chi(G)$ is finite, then \[\chi_{n,0}(G) = \chi(G).\]
\end{theorem}

\begin{proof} 
    It suffices to provide a coloring that shows $\chi_{n,0}(G) \leq \chi(G)$. For this, choose a minimal order proper labeling $\ell:V\longrightarrow \Z$ of $G$. Define a new labeling $\ell'$ of $G$ by $\ell'(v) = n\ell(v)$ for each $v\in V$. As this is a proper closed coloring with remainder $0\mod n$ of $G$, we are done.
\end{proof}

Accordingly, for $\chi_{n,k}(G)$, we will often only consider the case of $k\not\equiv 0 \mod n$ for the rest of this paper.

By canceling common summands, we immediately get the following result on symmetric differences. 
\begin{lemma}\label{lem: venn}
    If $\ell$ is a closed coloring with remainder $k\mod n$ of $G=(V,E)$ and $v,w\in V$, then
    \[ \sum_{u\in N[v]\backslash N[w]} \ell(u)
        \equiv \sum_{u\in N[w]\backslash N[v]} \ell(u) \mod n. \]
\end{lemma}

Next is a result on elementary operations.
\begin{theorem}\label{thm: k mod n units and divisor first relations}
    Let $k,u,v,d,c,k_1,k_2\in\Z$ and $n\in\Z^+$. If the right-hand side of each displayed equation below exists, we have the following:
\begin{itemize}    
    \item If $[u]$ is a unit in $\Z_n^\times$, then
        \[ \chi_{n,uk}(G) = \chi_{n,k}(G). \] 
    \item More generally, 
        \[ \chi_{n,vk}(G) \leq \chi_{n,k}(G). \] 
    \item If $d$ is a common divisor of $k$ and $n$, then
        \[ \chi_{n,k}(G) \leq \chi_{\frac{n}{d}, \frac{k}{d}}(G).\]
    \item If $d$ divides $n$, then
        \[ \chi_{ \frac{n}{d}, k}(G) \leq \chi_{n,k}(G).\]
    \item If $G$ admits a constant closed coloring with remainder $c \mod n$, then
        \[ \chi_{n,k - c}(G) = \chi_{n,k}(G).\]
    \item Finally,
        \[ \chi_{n,k_1+k_2}(G) \leq \chi_{n,k_1}(G) \chi_{n,k_2}(G).\]
\end{itemize}
\end{theorem}

\begin{proof}
    For the fourth statement, 
    let $\ell$ be a minimal order proper closed coloring with remainder $k\mod n$ of $G$. As this is also a proper closed coloring with remainder $k\mod \frac{n}{d}$ of $G$, we are done.
    For the third statement, let $\ell$ be a minimal order proper closed coloring with remainder $\frac{k}{d}\mod \frac{n}{d}$ of $G$. Define a new coloring $\ell'$ of $G$ by $\ell'(v) = d\ell(v)$ for each $v\in V$. As this is a proper closed coloring with remainder $k\mod n$ of $G$, we are done.
    The first statement follows by multiplying appropriate closed colorings of $G$ by $u$ or its inverse $\mod n$, and the second statement follows similarly.
    For the fifth statement, note that adding and subtracting the constant closed coloring leads from any minimal order proper closed coloring with remainder $k\mod n$ of $G$ to proper closed colorings of $G$ with remainders $(k+c)\mod n$ and $(k-c) \mod n$, respectively. For the last statement, let $\ell_1$ and $\ell_2$ be minimal order proper closed colorings of $G$ with remainders $k_1\mod n$ and $k_2\mod n$, respectively. Fix any injective map $\iota:\Z\times \Z \to \Z$ such that $\iota(z_1,z_2) \equiv (z_1+z_2) \mod n$ for all $z_1,z_2\in \Z$, and define $\ell'(v) = \iota(\ell_1(v), \ell_2(v))$ for each $v\in V$ for a proper closed coloring $\ell'$ with remainder $(k_1+k_2)\mod n$ of $G$.
\end{proof}

Our bound on $\chi_{n,k_1+k_2}(G)$ in the last statement of Theorem \ref{thm: k mod n units and divisor first relations} seems rather rough. In particular, it is natural to ask the following question:

\begin{question}
    Let $k_1,k_2\in\Z$ and $n\in\Z^+$, and let $\chi_{n,k_1}(G)$ and $\chi_{n,k_2}(G)$ exist. Does this imply
    $\chi_{n,k_1+k_2}(G) \leq \chi_{n,k_1}(G) +\chi_{n,k_2}(G)$?
\end{question}
          
Next we turn to a theorem on existence. We will see below, in Theorem \ref{thm: k mod n calc for C_m},
that $\chi_{n,k}(G)$ need not exist.


\begin{theorem}\label{thm: remainder k mod n existence criteria and upper bound by n chi}
    Let $k\in\Z$ and $n\in\Z^+$, and let $\chi(G)$ be finite. Then a proper closed coloring with remainder $k\mod n$ of $G$ exists if and only if a closed coloring with remainder $k\mod n$ of $G$ exists. In that case,
    \[ \chi_{n,k}(G) \leq n \, \chi(G). \]
    More precisely, if $\ell$ is a closed coloring with remainder $k \mod n$ of $G$, then
    \[ \chi_{n,k} \leq |\ell| \, \chi(G).\]
\end{theorem}

\begin{proof}
    Let $\ell$ be a closed coloring with remainder $k\mod n$ of $G$ and let $\ell'$ be a minimal proper labeling of $G$. We may assume that the range of $\ell$ sits in $[0,n-1]$, and we may assume that the range of $\ell'$ sits in $n\Z$. Then the labeling $\ell + \ell'$ is a proper closed coloring with remainder $k\mod n$ of $G$. As its order is bounded by $|\ell| \chi(G)$ and since $|\ell|\leq n$, we are done.
\end{proof} 

For our next discussion, we recall the definition of an efficient dominating set from \cite[Section 3]{bakkerIEDS}.

\begin{definition}\label{def: independent and dominating sets}
    Let $U\subseteq V$ for a graph, $G =(V,E)$. We say that $U$ is
    \begin{itemize}
        \item an \emph{efficient dominating set} if $|N(v)\cap U|=1$ for every $v\in V\backslash U$.
        \item an \emph{independent efficient dominating set} (IEDS) if $|N[v]\cap U|=1$ for every $v\in V$, i.e., it is an independent set and an efficient dominating set.
    \end{itemize}
    We say that a graph $G$ \emph{admits an IEDS} if such a collection of vertices exists for $G$.
\end{definition}

It has been shown by Bakker and van Leeuwen \cite[Theorem 3.3]{bakkerIEDS} that determining whether an arbitrary graph $G$ admits an IEDS is NP-complete. In the same paper, they also provide a linear-time algorithm that determines whether any given finite tree admits an IEDS. 

Notice that our Theorem \ref{thm: k mod n units and divisor first relations} shows that $\chi_{n,k}(G) < \infty$ for all $k \in \Z$ if and only if $\chi_{n,1}(G) < \infty$. The following question asks if this is nearly equivalent to determining whether $G$ admits an IEDS.

\begin{question}\label{conj: IEDS iff k,n finite}
    For a graph $G$, does $\chi_{n,k}(G) < \infty$ hold for all $k \in \Z$ and $n\in\Z^+$ if and only if \mbox{$\chi(G)<\infty$} and $G$ admits an IEDS?
\end{question}

Lemma \ref{lem: Colorings given IEDS} proves the backwards direction of this question.  
The condition that $\chi(G)<\infty$ is necessary as $K_\infty$ admits an IEDS, via a single vertex, and $\chi(K_\infty)=\infty$, but $\chi_{n,k}(G) < \infty$ fails.


\begin{lemma}\label{lem: Colorings given IEDS}
    If $G=(V,E)$ admits an IEDS $U\subseteq V$ and $\chi(G)<\infty$, then $\chi_{n,k}(G)$ exists for all $k\in \Z$ and $n\in\Z^+$. In particular, 
    \[ \chi(G) \leq \chi_{n,k}(G) \leq \chi(G)+1.\]
    If $U$ can be colored with a single color in some minimal proper labeling of $G$ such that $U$ contains all vertices of that color, then the inequality improves to \[\chi_{n,k}(G) = \chi(G).\]
    
\end{lemma}

\begin{proof}
    Let $U$ be an IEDS for $G$.
    Write $\ell$ for a minimal proper labeling of $G$ and suppose its range lies in $n\Z \cap (k, \infty)$. The proof is finished by defining a closed coloring $\ell'$ with remainder $k\mod n$ of $G$ via
    \[
        \ell'(v) = 
        \begin{cases}
            \ell(v) & \text{if } v\in V\backslash U\\
            k & \text{if } v\in U.\qquad\qedhere
        \end{cases}
    \]
\end{proof}

\section{Finite Order Examples}\label{sec: finite order exs}

We begin with the \emph{complete graph on $m$ vertices}, $K_m$, the \emph{star on $m+1$ vertices}, $S_m$, and the \emph{friendship graph}, $F_m$, consisting of $m$ copies of $C_3$ joined at a single vertex.

\begin{theorem}
    Let $k\in \Z$ and $n,m\in\Z^+$. Then 
    \[ \chi_{n,k}(K_m) = m, \]
    \[ \chi_{n,k}(S_m) = 2,\]
    \[ \chi_{n,k}(F_m) = 3.\]
\end{theorem}

\begin{proof}
    These results follow from Lemma \ref{lem: Colorings given IEDS} with the IEDS consisting of a single vertex, respectively.
\end{proof}

Next we turn to the \emph{path on $m$ vertices}, $P_m$.

\begin{theorem}
    Let $k\in \Z$ and $n,m\in\Z^+$ with $k \not\equiv 0 \mod n$. Then 
    \[ \chi_{n,k}(P_2) = \chi_{n,k}(P_3) = 2 \]
    and 
    \[ \chi_{n,k}(P_m) = 3 \]
    for $m \geq 4$.
\end{theorem}

\begin{proof}
    The first set of equalities is straightforward using proper closed colorings of $(0,k)$ and $(0,k,0)$, respectively.
    
    For $m\geq 4$, we first show that $\chi_{n, k}(P_m) > 2$. 
    If not, there is a proper closed $2$-coloring with remainder $k\mod n$ of the form $(a,b,a,b,\ldots)$.    However, Lemma \ref{lem: venn}, applied to the first two vertices, forces $a \equiv 0 \mod n$ and,
    applied to the second and third vertices, forces $b \equiv a \mod n$. As this requires $k\equiv 0 \mod n$, we obtain a contradiction. 

    It remains to exhibit a proper closed $3$-coloring of $P_m$ with remainder $k \mod n$.
    If $m \equiv 1 \mod 3$, then one such coloring is provided by $(k,0,n,k,0,n,\ldots, 0,n,k)$.
    If $m \not\equiv 1 \mod 3$, then $(0,k,n,0,k,n,\ldots)$ works.
\end{proof}




Next, we turn to the \emph{complete bipartite graph}, $K_{i,j}$, with parts of sizes $i$ and $j$.

\begin{theorem} Let $k\in \Z$ and $i,j,n \in \Z^+$. Then
$\chi_{n,k}(K_{i,j})$ exists if and only if 
\[ (ij-1,n) \mid  (j-1)k. \]
In that case, 
    \[
        \chi_{n,k}(K_{i,j}) = 2.
    \]
    Note that the condition $(ij-1,n)\mid  (j-1)k$ is equivalent to $(ij-1,n) \mid  (i-1)k$.
\end{theorem}
\begin{proof}
    \indent Let $V_1$ and $V_2$ with $|V_1|=i$ and $|V_2|=j$ denote the vertex sets belonging to the two parts of $K_{i,j}$. 
    If a closed coloring with remainder $k\mod n$ of $G$ exists, 
    Lemma \ref{lem: venn}, applied to any two vertices in the same part shows that the labels are congruent $\mod n$. Therefore, $\chi_{n,k}(K_{i,j})$ exists if and only if it is $2$. 
    
    Write $\alpha$ and $\beta$ for the shared label of the vertices in $V_1$ and $V_2$, respectively. 
    There exists a closed coloring with remainder $k\mod n$ of $G$ if and only if there exist solutions for $\alpha, \beta$ to the equations 
    \[i\alpha + \beta \equiv \alpha + j\beta \equiv k \mod n.\]
    In turn, this is equivalent to setting $\beta \equiv (k - i\alpha) \mod n$ and requiring a solution to the equation
    \[ (ij-1)\alpha \equiv (j-1)k \mod n. \]
    As a result, $\chi_{n,k}(K_{i,j})$ exists if and only if 
    $(ij-1,n) \mid  (j-1)k$. As $(ij-1)k = (i-1)(j-1)k+(i-1)k+(j-1)k$, we see that this condition is equivalent to $(ij-1,n) \mid  (i-1)k$.   
    %
    %
\end{proof}


We turn now to \emph{regular graphs}.

\begin{theorem}\label{thm: regular chi k n}
    Let $k\in \Z$ and $n,j\in\Z^+$, and let $G$ be a $j$-regular graph. Then
    \[ (j+1, n) \mid k \implies \chi_{n,k}(G) = \chi(G) \]
    and, if $G$ is finite, 
    \[ (j+1, n) \nmid k|V| 
        \implies \chi_{n,k}(G) \text{ does not exist.} \]
\end{theorem}

\begin{proof}
    If $(j+1, n) \mid k$, then $(j+1)x \equiv k \mod n$ can be solved. In that case, a constant labeling of $G$ by $x$ is a closed coloring with remainder $k \mod n$. Furthermore, note that $\chi(G)\le j+1$ for any $j$-regular graph~$G$. Theorem \ref{thm: remainder k mod n existence criteria and upper bound by n chi} finishes the proof.

    Now suppose there is a closed coloring $\ell$ of $G$ with remainder $k \mod n$, but $(j+1, n) \nmid k|V|$. 
    Let 
    \[ S = \sum_{v\in V} \sum_{u\in N[v]} \ell(u). \]
    Then, $S \equiv k|V| \mod n$ as $\sum_{u\in N[v]} \ell(u) \equiv k \mod n$ for all $v\in V$. 
    But each $v\in V$ is in exactly $j+1$ closed neighborhoods. Therefore, 
    $S = (j+1)\sum_{v\in V} \ell(v)$. As a result, the equation $(j+1)x \equiv k|V| \mod n$ can be solved. As this happens if and only if $(j+1,n)\mid k|V|$, we are done. 
\end{proof}

We turn now to the \emph{cycle on $m$ vertices}, $C_m$. Recall that $\chi(C_m)$ is $2$ when $m$ is even and $3$ when $m$ is odd.

\begin{theorem}\label{thm: k mod n calc for C_m}
    Let $k\in \Z$ and $n,m\in\Z^+$ with $m\geq 3$. Then
    \[
        \chi_{n,k}(C_m) =
        \begin{cases}
            2 & \text{if } 
                (3,n) \mid k \text{ and } 2 \mid m,\\
            3 & \text{if }  
               (3,n)\mid k \text{ and } 2 \nmid m \text { or}\\
             & \text{if }  
               (3,n) \nmid k \text{ and } 3 \mid m, \\
            \text{does not exist} & \text{if } 
                (3,n) \nmid k \text{ and } 3 \nmid m. 
        \end{cases}
    \]
\end{theorem}

\begin{proof}
    Theorem \ref{thm: regular chi k n} shows that $\chi_{n,k}(C_m) =  \chi(C_m)$ when $(3,n) \mid k$ and that $\chi_{n,k}(C_m)$ does not exist when $(3,n) \nmid km$, which is equivalent to $(3,n) \nmid k$ and $3\nmid m$. 
    If we are outside of either of these two cases, then $(3,n) \nmid k$ and $3\mid m$. In that case, since $3\mid m$, $\chi_{n,k}(C_m) \leq 3$ as demonstrated by the closed coloring $(0,k,n,0,k,n,\ldots)$. 
    
    However, as $\chi(G)\leq \chi_{n,k}(G)$ for all graphs, $\chi_{n,k}(C_m)$ can possibly be $2$ only when $m$ is also even. In this case, in the standard manner, denote the vertices of $C_m$ by $v_i$ for $i\in\Z_m$. Suppose $\ell$ is a proper closed $2$-coloring with remainder $k\mod n$.  Lemma \ref{lem: venn}, applied to adjacent vertices, shows that $\ell(v_i)\equiv\ell(v_{i+3}) \mod n$. The proper $2$-coloring forces $\ell(v_i)\equiv\ell(v_{i+2})\mod n$. As a result, the closed coloring is constant $\mod n$. This means that $3x \equiv k \mod n$ has a solution. In turn, this means that $(3,n)\mid k$, which is not possible in this case. 
\end{proof}

\section{Infinite Order Examples}\label{sec: infinite order exs}

Next, we turn to the \emph{complete $m$-ary rooted tree of infinite height},~$T_m$.

\begin{theorem}\label{thm: trees}
    Let $k\in \Z$ and $n,m\in\Z^+$. Then
    \[
        \chi_{n,k}(T_m)=
        \begin{cases}
            2 & \text{if } n \mid mk, \\
            3 & \text{else.}
        \end{cases}
    \]
\end{theorem}

\begin{proof}
    For a vertex $v$ of $T_m$, write $h(v)$ for the \emph{height} of $v$, i.e., the distance from a vertex $v$ to the root, $v_0$. 
    If $n\mid mk$, then the labeling $\ell$ on the vertices of $T_m$ given by
    \[
        \ell(v) =
        \begin{cases}
            0 & \text{if } h(v) \text{ is odd}, \\
            k & \text{otherwise,}
        \end{cases}
    \]
    gives a proper closed coloring with remainder $k\mod n$. Thus $\chi_{n,k}(T_m) = 2$ in this case.

    In fact, if $\chi_{n,k}(T_m) = 2$, induction shows that any proper closed $2$-coloring of $T_m$ must be constant on vertices of the same height. Therefore, there exist $\alpha, \beta\in\Z$ so that
    \begin{align*}
        \alpha + m\beta &\equiv k \mod n,\\
        \beta + (m+1)\alpha &\equiv k \mod n,\\
        \alpha + (m+1)\beta &\equiv k \mod n.
    \end{align*}
    In turn, the first and third displayed equations force $\beta\equiv 0\mod n$. The first then yields $\alpha\equiv k \mod n$, and the second gives $mk\equiv 0 \mod n$.

    We finish the proof by showing $\chi_{n,k}(T_m) \le 3$ with the help of Lemma~\ref{lem: Colorings given IEDS}. 
    Define the IEDS $U$ inductively via the height of a vertex $v$: let $v_0\not\in U$. Then, starting with $v:=v_0$,
    if neither $v$ nor its parent (if it exists) lies in $U$, have $U$ contain exactly one child of $v$. Otherwise, have $U$ contain no children of $v$. It is straightforward to check that this is an IEDS.
\end{proof}

\begin{remark}
Note that the above proof also applies to show that $\chi_{n,k}(T) \leq 3$ for any (infinite) tree $T$ without any leaves. This result differs significantly from our later findings on finite trees, see Theorems~\ref{thm: baby tree example} and \ref{thm: perfect binary trees}.
\end{remark}

Next we look at the \emph{regular, infinite tilings of the plane}. Write $R_3$, $R_4$, and $R_6$ for the tilings by regular triangles, squares, and hexagons, respectively.
\begin{theorem}\label{thm: regular infinite tilins of plane x n k}
    For the regular, infinite tilings of the plane,
    \[
    \chi_{n,k} (R_3) = \begin{cases}
        3 & \text{if } (7,n)\mid k,\\
        4 & \text{else,}
    \end{cases}
    \]
    \[
    \chi_{n,k} (R_4) = \begin{cases}
        2 & \text{if } (5,n)\mid k, \\
        3 & \text{else,}
    \end{cases}
    \]
    and
    \[
    \chi_{n,k} (R_6) = \begin{cases}
        2 & \text{if } (8,n)\mid 2k,\\
        3 & \text{else.}
    \end{cases}
    \]
\end{theorem}

\begin{proof}
    First of all, $R_3$, $R_4$, and $R_6$ all admit IEDS. See Figures \ref{fig: IEDS for triangular tiling}, \ref{fig: IEDS for square tiling}, and \ref{fig: IEDS for hex tiling}, respectively, where the IEDS is given by the diamond vertices. Lemma \ref{lem: Colorings given IEDS} therefore shows that $\chi_{n,k}(G)$ is bounded by $\chi(G)+1$ for each of these graphs. Recall that $\chi(R_3)=3$ and $\chi(R_4)=\chi(R_6)=2$.

    Begin with $R_3$. Theorem \ref{thm: regular chi k n} shows that $\chi_{n,k}(R_3)=3$ if $(7,n)\mid k$.  Conversely, if $\chi_{n,k}(R_3)=3$, there exist $\alpha, \beta, \gamma \in\Z$ so that
    \begin{align*}
       \alpha + 3\beta + 3\gamma &\equiv k \mod n,\\
       3\alpha + \beta + 3\gamma &\equiv k \mod n,\\
       3\alpha + 3\beta + \gamma &\equiv k \mod n. 
    \end{align*}
    Adding these equations shows that $7(\alpha+\beta+\gamma)\equiv 3k \mod n$. Therefore $(7,n)\mid 3k$. As this is equivalent to $(7,n)\mid k$, we are done.

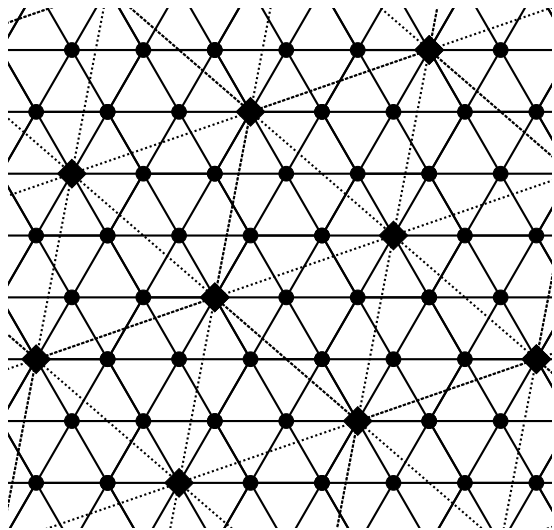
\begin{figure}[H]
\begin{center}
\begin{tikzpicture}[scale=0.95]

\begin{scope}[local bounding box=L]

\draw[-, color=white] (0.125,0.25) -- (0.125,7.5) -- (7.7,7.5) -- (7.7,0.25) -- (0.125,0.25);

\end{scope}
\clip (current bounding box.south west) rectangle (current bounding box.north east);

\begin{scope}

    \foreach \i in {-1,...,2} 
    \foreach \j in {-1,...,4} {
  
    \foreach \a in {0,120,-120} \draw[color=black, thick] (3*\i,2*sin{60}*\j) -- +(\a:1);
    \foreach \a in {0,120,-120} \draw[color=black, thick] (3*\i+3*cos{60},2*sin{60}*\j+sin{60}) -- +(\a:1);

    \foreach \a in {0,120,-120} \draw[color=black, thick] (1/2 + 3*\i,2*sin{60}*\j + 0.866) -- +(\a:1);
    \foreach \a in {0,120,-120} \draw[color=black, thick] (1/2 + 3*\i+3*cos{60},2*sin{60}*\j+sin{60} + 0.866) -- +(\a:1);

    \foreach \a in {0,120,-120} \draw[color=black, thick] (3*\i,2*sin{60}*\j) -- +(\a:-1);
    \foreach \a in {0,120,-120} \draw[color=black, thick] (3*\i+3*cos{60},2*sin{60}*\j+sin{60}) -- +(\a:-1);

    \foreach \a in {0,120,-120} \draw[color=black, thick] (1/2 + 3*\i,2*sin{60}*\j + 0.866) -- +(\a:-1);
    \foreach \a in {0,120,-120} \draw[color=black, thick] (1/2 + 3*\i+3*cos{60},2*sin{60}*\j+sin{60} + 0.866) -- +(\a:-1);
  
    }

    \foreach \i in {-1,...,2} 
    \foreach \j in {-1,...,4} {
  
    \node[shape=circle,scale=0.5,draw=black, fill=black] (H) at (3*\i,2*sin{60}*\j) {};
    \node[shape=circle,scale=0.5,draw=black, fill=black] (H) at (3*\i+3*cos{60},2*sin{60}*\j+sin{60}) {};
    \node[shape=circle,scale=0.5,draw=black, fill=black] (H) at (3*\i-cos{60},2*sin{60}*\j+sin{60}) {};  
    \node[shape=circle,scale=0.5,draw=black, fill=black] (H) at (3*\i+3*cos{60}-cos{60},2*sin{60}*\j) {};

    \node[shape=circle,scale=0.5,draw=black, fill=black] (H) at (3*\i,2*sin{60}*\j) {};
    \node[shape=circle,scale=0.5,draw=black, fill=black] (H) at (3*\i - 3*cos{60},2*sin{60}*\j+sin{60}) {};
    \node[shape=circle,scale=0.5,draw=black, fill=black] (H) at (3*\i + cos{60},2*sin{60}*\j+sin{60}) {};  
    \node[shape=circle,scale=0.5,draw=black, fill=black] (H) at (3*\i - 3*cos{60} + cos{60},2*sin{60}*\j) {};

    }
    
\end{scope}


\begin{scope}[shift={(0,0)}, scale = 2.646, rotate = 19.11]

    \foreach \i in {-1,...,2} 
    \foreach \j in {-1,...,4} {
  
    \foreach \a in {0,120,-120} \draw[color=black, densely dotted, thick] (3*\i,2*sin{60}*\j) -- +(\a:1);
    \foreach \a in {0,120,-120} \draw[color=black, densely dotted, thick] (3*\i+3*cos{60},2*sin{60}*\j+sin{60}) -- +(\a:1);

    \foreach \a in {0,120,-120} \draw[color=black, densely dotted, thick] (1/2 + 3*\i,2*sin{60}*\j + 0.866) -- +(\a:1);
    \foreach \a in {0,120,-120} \draw[color=black, densely dotted, thick] (1/2 + 3*\i+3*cos{60},2*sin{60}*\j+sin{60} + 0.866) -- +(\a:1);

    \foreach \a in {0,120,-120} \draw[color=black, densely dotted, thick] (3*\i,2*sin{60}*\j) -- +(\a:-1);
    \foreach \a in {0,120,-120} \draw[color=black, densely dotted, thick] (3*\i+3*cos{60},2*sin{60}*\j+sin{60}) -- +(\a:-1);

    \foreach \a in {0,120,-120} \draw[color=black, densely dotted, thick] (1/2 + 3*\i,2*sin{60}*\j + 0.866) -- +(\a:-1);
    \foreach \a in {0,120,-120} \draw[color=black, densely dotted, thick] (1/2 + 3*\i+3*cos{60},2*sin{60}*\j+sin{60} + 0.866) -- +(\a:-1);
  
    }

    \foreach \i in {-1,...,2} 
    \foreach \j in {-1,...,4} {
  
    \node[shape=diamond,scale=0.7,draw=black, fill=black] (H) at (3*\i,2*sin{60}*\j) {};
    \node[shape=diamond,scale=0.7,draw=black, fill=black] (H) at (3*\i+3*cos{60},2*sin{60}*\j+sin{60}) {};
    \node[shape=diamond,scale=0.7,draw=black, fill=black] (H) at (3*\i-cos{60},2*sin{60}*\j+sin{60}) {};  
    \node[shape=diamond,scale=0.7,draw=black, fill=black] (H) at (3*\i+3*cos{60}-cos{60},2*sin{60}*\j) {};

    \node[shape=diamond,scale=0.7,draw=black, fill=black] (H) at (3*\i,2*sin{60}*\j) {};
    \node[shape=diamond,scale=0.7,draw=black, fill=black] (H) at (3*\i - 3*cos{60},2*sin{60}*\j+sin{60}) {};
    \node[shape=diamond,scale=0.7,draw=black, fill=black] (H) at (3*\i + cos{60},2*sin{60}*\j+sin{60}) {};  
    \node[shape=diamond,scale=0.7,draw=black, fill=black] (H) at (3*\i - 3*cos{60} + cos{60},2*sin{60}*\j) {};

    }

\end{scope}
\end{tikzpicture}
\end{center}
\caption{IEDS (Diamonds) for the Triangular Tiling of the Plane}
\label{fig: IEDS for triangular tiling}
\end{figure}

    Next, turn to $R_4$. Theorem \ref{thm: regular chi k n} shows that $\chi_{n,k}(R_4)=2$ if $(5,n)\mid k$.  Conversely, suppose $\chi_{n,k}(R_4)=2$. Then there exist $\alpha, \beta\in\Z$ so that 
    \begin{align*}
        \alpha + 4\beta & \equiv k \mod n,\\
        4\alpha +\beta  & \equiv k \mod n.
    \end{align*}
    Adding these equations shows that $5(\alpha+\beta)\equiv 2k \mod n$. Therefore $(5,n)\mid 2k$. As this is equivalent to $(5,n)\mid k$, we are done.

    Finally, consider $R_6$. Here, $\chi_{n,k}(R_6)=2$ if and only if there exist $\alpha, \beta\in\Z$ so that
    \begin{align*}
        \alpha + 3\beta & \equiv k \mod n,\\
        3\alpha +\beta & \equiv k \mod n.
    \end{align*}
    Multiplying the top equation by $3$ and subtracting the bottom equation implies that $8\beta \equiv 2k \mod n$. Therefore, $(8,n)\mid 2k$ is necessary. Conversely, if $(8,n)\mid 2k$, let $\beta$ be a solution to $8\beta \equiv 2k \mod n$ and define $\alpha = k - 3\beta.$ Then 
    \[ 3\alpha +\beta = 3k -8\beta \equiv k \mod n\]
    and we are done. 
\end{proof}


\begin{figure}[H]
\begin{center}
\begin{tikzpicture}[scale=0.73]

\begin{scope}[local bounding box=L]
    \foreach \i in {0,...,8}
    \foreach \j in {0,...,8} {

    \node[shape=circle,scale=0.5, draw=black, fill=black] (\i,\j) at (\i,\j) {};   

    }

    \foreach \i in {0,...,8}
    \foreach \j in {0,...,8} {
    
    \path [-, color=black, thick] (\i,\j) edge node[left] {} (\i+1,\j);
    \path [-, color=black, thick] (\i,\j) edge node[left] {} (\i,\j+1);

    \path [-, color=black, thick] (\i,\j) edge node[left] {} (\i-1,\j);
    \path [-, color=black, thick] (\i,\j) edge node[left] {} (\i,\j-1);

    }
\end{scope}


\clip (current bounding box.south west) rectangle (current bounding box.north east);
\begin{scope}[shift={(L.center)}, scale = sqrt(5), rotate = 26.57]
    \foreach \i in {-2,...,2}
    \foreach \j in {-2,...,2} {

    \node[shape=diamond,scale=0.7, draw=black, fill=black] (\i,\j) at (\i,\j) {};   

    }

    \foreach \i in {-2,...,2}
    \foreach \j in {-2,...,2} {
    
    \path [-, color=black, densely dotted, thick] (\i,\j) edge node[left] {} (\i+1,\j);
    \path [-, color=black, densely dotted, thick] (\i,\j) edge node[left] {} (\i,\j+1);

    \path [-, color=black, densely dotted, thick] (\i,\j) edge node[left] {} (\i-1,\j);
    \path [-, color=black, densely dotted, thick] (\i,\j) edge node[left] {} (\i,\j-1);

    }
\end{scope}
\end{tikzpicture}
\end{center}
\caption{IEDS (Diamonds) for the Square Tiling of the Plane}
\label{fig: IEDS for square tiling}
\end{figure}
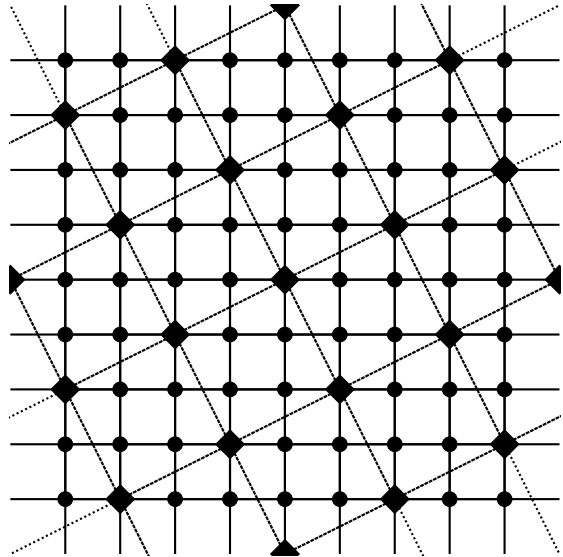

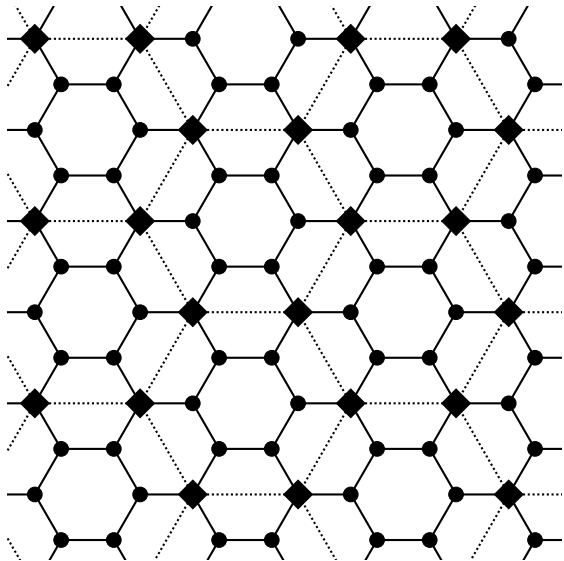
\begin{figure}[H]
\begin{center}
\begin{tikzpicture}[scale=0.7]

\begin{scope}[local bounding box=L]

\draw[-, color=white] (0.5,0.5) -- (0.5,11) -- (11,11) -- (11,0.5) -- (0.5,0.5);

\end{scope}
\clip (current bounding box.south west) rectangle (current bounding box.north east);
\begin{scope}[local bounding box=L]

    \foreach \i in {0,...,4} 
    \foreach \j in {0,...,8} {
  
    \foreach \a in {0,120,-120} \draw[thick] (3*\i,2*sin{60}*\j) -- +(\a:1);
    \foreach \a in {0,120,-120} \draw[thick] (3*\i+3*cos{60},2*sin{60}*\j+sin{60}) -- +(\a:1);
  
    }

    \foreach \i in {0,...,4} 
    \foreach \j in {0,...,8} {
  
    \node[shape=circle,scale=0.5, draw=black, fill=black] (H) at (3*\i,2*sin{60}*\j) {};
    \node[shape=circle,scale=0.5, draw=black, fill=black] (H) at (3*\i+3*cos{60},2*sin{60}*\j+sin{60}) {};
    \node[shape=circle,scale=0.5, draw=black, fill=black] (H) at (3*\i-cos{60},2*sin{60}*\j+sin{60}) {};  
    \node[shape=circle,scale=0.5, draw=black, fill=black] (H) at (3*\i+3*cos{60}-cos{60},2*sin{60}*\j) {};
    
    }
    
\end{scope}

\begin{scope}[shift={(1,0)}, scale = 2]

    \foreach \i in {-4,...,4} 
    \foreach \j in {-8,...,8} {
  
    \foreach \a in {0,120,-120} \draw[color=black,densely dotted, thick] (3*\i,2*sin{60}*\j) -- +(\a:1);
    \foreach \a in {0,120,-120} \draw[color=black,densely dotted, thick] (3*\i+3*cos{60},2*sin{60}*\j+sin{60}) -- +(\a:1);
  
    }

    \foreach \i in {-4,...,4} 
    \foreach \j in {-8,...,8} {
  
    \node[shape=diamond,scale=0.7, draw=black, fill=black] (H) at (3*\i,2*sin{60}*\j) {};
    \node[shape=diamond,scale=0.7, draw=black, fill=black] (H) at (3*\i+3*cos{60},2*sin{60}*\j+sin{60}) {};
    \node[shape=diamond,scale=0.7, draw=black, fill=black] (H) at (3*\i-cos{60},2*sin{60}*\j+sin{60}) {};  
    \node[shape=diamond,scale=0.7, draw=black, fill=black] (H) at (3*\i+3*cos{60}-cos{60},2*sin{60}*\j) {};
    
    }

\end{scope}
\end{tikzpicture}
\end{center}
\caption{IEDS (Diamonds) for the Hexagonal Tiling of the Plane}
\label{fig: IEDS for hex tiling}
\end{figure}

\section{Trees}\label{sec: trees}

Next we turn to trees. Theorems \ref{thm: baby tree example} and \ref{thm: perfect binary trees} below show that existence of $\chi_{n,k}(G)$ can be very complicated and chaotic. By way of a toy example to illustrate complexity, consider first the \emph{caterpillar tree} $C_{m_1, m_2}$ in Figure \ref{thm: baby tree example} where there are $m_1$ legs underneath vertex $x$ and $m_2$ legs underneath vertex $y$.
    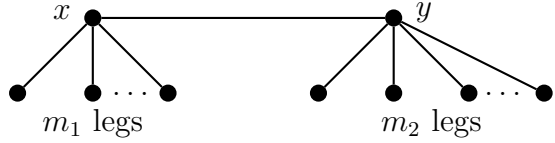
\begin{figure}[H]
        \centering
        \begin{tikzpicture}[thick]
          \node[circle, draw, fill, inner sep=2pt] (x) at (0,0) {};
          \node[circle, draw, fill, inner sep=2pt] (y) at (4,0) {};
          \node[left, xshift=-4pt, yshift=2pt] at (x) {$x$};
          \node[right, xshift=4pt, yshift=2pt] at (y) {$y$};
          \draw (x) -- (y);
          \foreach \i in {1,...,3}
            \node[circle, draw, fill, inner sep=2pt] (x\i) at (-\i+2,-1) {};
          \foreach \i in {1,...,3}
            \draw (x) -- (x\i);
          \foreach \i in {1,...,4}
            \node[circle, draw, fill, inner sep=2pt] (y\i) at (2+\i,-1) {};
          \foreach \i in {1,...,4}
            \draw (y) -- (y\i);
           \node at (0,-1.4) {$m_1$ legs};
           \node at (4.5,-1.4) {$m_2$ legs};
           \node at (.5,-1) {$\ldots$};
           \node at (5.45,-1) {$\ldots$};
        \end{tikzpicture}
        \caption{A Baby Tree}
        \label{figure: Illustrative Example, $C_{m_1, m_2}$}
    \end{figure}
    
\begin{theorem} \label{thm: baby tree example}
    For $k\in \Z$ and $n, m_1, m_2 \in\Z^+$, consider the caterpillar tree $C_{m_1, m_2}$ in Figure \ref{figure: Illustrative Example, $C_{m_1, m_2}$} where there are $m_1$ legs underneath vertex $x$ and $m_2$ legs underneath vertex $y$. Let
    \[ M := m_1m_2-m_1-m_2.\]
    Then $\chi_{n,k}(C_{m_1, m_2})$ exists if and only if 
    \[ (M,n) \mid m_1 k. \]
    This is equivalent to $(M,n) \mid m_2 k$. In this case, 
    \[
    \chi_{n,k} (C_{m_1, m_2}) = \begin{cases}
        2 & \text{if } n \mid \frac{m_1 m_2 k}{\gcd(m_1, m_2, n)}\,,\\
        3 & \text{if } n \nmid \frac{m_1 m_2 k}{\gcd(m_1, m_2, n)} \text { and } n\mid \frac{(m_1-m_2) k}{\gcd(m_1-2, m_2-2, n)}\,,\\
        4 & \text{else.}
    \end{cases}
    \]
\end{theorem}
        
\begin{proof}
    Using Lemma \ref{lem: venn} on the legs, it is easy to see that existence of a closed coloring $\ell$ with remainder $k \mod n$ requires that all legs under $x$ and all legs under $y$ share a common label $\mod n$, respectively.
    In particular, for a minimal closed coloring $\ell$, we may use the same label $\alpha_1\in \Z$ for all legs under $x$ and the same label $\alpha_2\in \Z$ for all legs under~$y$. Then $\ell(x)\equiv (k - \alpha_1) \mod n$ and $\ell(y)\equiv (k - \alpha_2) \mod n$, and if $\chi_{n,k} (C_{m_1, m_2})$ exists, then $\chi_{n,k} (C_{m_1, m_2}) \le 4$.

    We continue with some necessary conditions for the existence of a closed coloring.
    Lemma \ref{lem: venn} applied to $x$ and $y$ shows that $m_1\alpha_1 \equiv m_2 \alpha_2 \mod n$.
    The requirement $\sum_{v\in N[x]}\ell(v)\equiv k \mod n$ gives
    \[(k - \alpha_1)+(k - \alpha_2)+m_1\alpha_1 \equiv k \mod n,\]
    which simplifies to 
    \begin{align}
    \alpha_2 \equiv (m_1 - 1)\alpha_1 + k\, \mod n. \label{eq6.1}
    \end{align}
    Multiplying this equation by $m_2$ and using $m_1\alpha_1 \equiv m_2 \alpha_2 \mod n$ gives
    \begin{align}
    M\alpha_1 \equiv -m_2 k \mod n. \label{eq6.2}
    \end{align}
    This equation has a solution if and only if $(M,n) \mid m_2 k$. 
    Furthermore, as $Mk=m_1m_2k-m_1k-m_2k$, we see that this condition is equivalent to $(M,n) \mid m_1 k$.

    Conversely, if $(M,n) \mid m_2 k$, let $\alpha_1$ be a solution to the equation $M\alpha_1 \equiv -m_2 k \mod n$ as required by $\eqref{eq6.2}$ and let $\alpha_2 \equiv ((m_1 -1)\alpha_1 +k) \mod n$ as required by \eqref{eq6.1}. It is straightforward to verify that this results in $m_1\alpha_1 \equiv m_2 \alpha_2 \mod n$ and gives a closed coloring with remainder $k\mod n$. 
    As a result, we see that $\chi_{n,k}(C_{m_1, m_2})$ exists if and only if $(M,n) \mid m_2 k$. In this case, $\chi_{n,k}(C_{m_1, m_2})$ will be $4$ unless some of the labels $\alpha_1,\alpha_2, \ell(x), \ell(y)$ are congruent $\mod n$ and can be chosen to be equal. Thus, it remains to investigate the equations $\alpha_1 \equiv \ell(y) \mod n$, $\alpha_2 \equiv \ell(x) \mod n$, and $\alpha_1 \equiv \alpha_2\mod n$ for possible exceptional cases.
    
    In fact, it is straightforward to check that $\alpha_1 \equiv \ell(y) \mod n$ happens if and only if $\alpha_2 \equiv \ell(x) \mod n$ if and only if $m_1\alpha_1 \equiv 0 \mod n$, and we infer that $\chi_{n,k}(C_{m_1, m_2})$ will be $2$. This happens if and only if there is a solution to $m_1\alpha_1 \equiv 0 \mod n$ and $M\alpha_1\equiv -m_2 k \mod n$, where the latter equation simplifies to $m_2\alpha_1\equiv m_2 k \mod n$. Therefore $\alpha_1 = n' j$, where $n':=\frac{n}{(n,m_1)}$ and $j\in\Z$ satisfy $m_2 n' j \equiv m_2 k \mod n$. In turn, this happens if and only if 
    $(m_2 n', n) \mid m_2 k$. As $(m_2 n', n) = n' (m_2, (n, m_1)) = n' \gcd(m_1, m_2, n)$, the solution exists if and only if $n\mid \frac{(n,m_1) m_2 k}{\gcd(m_1, m_2, n)}$ if and only if $n\mid \frac{m_1 m_2 k}{\gcd(m_1, m_2, n)}$\,.

    It remains to consider the case $\alpha_1 \equiv \alpha_2\mod n$ which allows for closed $3$-colorings of $C_{m_1, m_2}$. As in the previous case, it is easy to check that $\alpha_1 \equiv \alpha_2\mod n$ happens if and only if  $(m_1-2)\alpha_1 \equiv -k \mod n$. This happens if and only if there is a solution to $(m_1-2)\alpha_1 \equiv -k \mod n$ and $M\alpha_1\equiv -m_2 k \mod n$, where the latter equation simplifies with the first one to $(m_1-m_2)\alpha_1\equiv 0 \mod n$. Therefore $\alpha_1 = n' j$, where $n':=\frac{n}{(n,m_1-m_2)}$ and $j\in\Z$ satisfy $(m_1-2) n' j \equiv -k \mod n$. In turn, this happens if and only if 
    $((m_1-2) n', n) \mid k$. As
    \begin{align*}
    ((m_1 -2) n', n) & = n' (m_1-2, (n, m_1-m_2))\\
    & = n' \gcd(m_1-2, n, m_1-m_2)\\ & = n' \gcd(m_1-2, m_2-2, n),
    \end{align*}
    the solution exists iff $n\mid \frac{(n,m_1-m_2) k}{\gcd(m_1-2, m_2-2, n)}$ iff $n\mid \frac{(m_1-m_2) k}{\gcd(m_1-2, m_2-2, n)}$\,.
\end{proof}

Next, turn to the \emph{rooted perfect binary tree of height $d$}, written~$T_{2,d}$.

\begin{lemma}\label{lem: perfect binary tree even same parity labels}
    Let $k\in \Z$ and $n,d\in\Z^+$. If $T_{2,d}$ admits a closed coloring with remainder $k\mod n$ and $n$ is even, then within each level of the tree $T_{2,d}$ the used integer labels of this coloring share the same parity.
\end{lemma}

\begin{proof} Suppose $\ell$ is a closed coloring of $T_{2,d}$ with remainder $k\mod n$.
    Start at the bottom, see Figure \ref{fig: bottom perfect binary tree}. Any leaves that share a parent have identical labels $\mod n$ by Lemma \ref{lem: venn} and, in particular, the same parity as $n$ is even. The requirement of being a closed coloring with remainder $k\mod n$ at the leaves means that the parity of the parent's label is shifted from the parity of the labels at the leaves by $k$, meaning that it retains the same parity if $k$ is even and switches the parity if $k$ is odd.
    \begin{figure}[H]
        \centering
        \begin{tikzpicture}
            \draw (0,0) -- (90:1);
            \fill (90:1) circle[radius=2pt];
            
            \draw (0,0) -- (270+50/2:1);
            \fill (270+50/2:1) circle[radius=2pt];
            
            \draw (0,0) -- (270-50/2:1);
            \fill (270-50/2:1) circle[radius=2pt];
            
            \fill (0,0) circle[radius=2pt];
            
            \node[left] at (270-50/2:1) {$[\alpha]$};
            
            \node[right] at (270+50/2:1) {$[\alpha]$};
            
            \node[right] at (0,0) {$[k-\alpha]$};
        \end{tikzpicture}
        \caption{Pair of Leaves Sharing a Parent in $T_{2,d}$}
        \label{fig: bottom perfect binary tree}
    \end{figure}
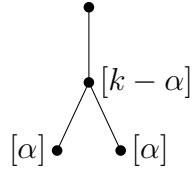
    
    Moving up a level, see Figure \ref{fig: mergeing of leaves in perfect binary} for when two pairs of leaves from the previous paragraph merge at the grandparent of each. The requirement of being a closed coloring with remainder $k\mod n$ at each of the parents forces the parity of the grandparent's label to be shifted from the parity of the respective parent's label again by $k$. As the coloring is consistent, this forces backwards parity equality on the leaves and their parents.
    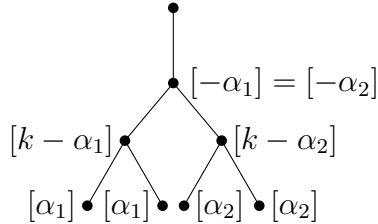
\begin{figure}[H]
        \centering
         \begin{tikzpicture}
            \draw (0,0) -- (90:1);
            \fill (90:1) circle[radius=2pt];
            
            \draw (0,0) -- (270+80/2:1);
            \fill (270+80/2:1) circle[radius=2pt];
            \draw (270+80/2:1) -- ++(270+30:1);
            \draw (270+80/2:1) -- ++(270-30:1);
            \fill (270+80/2:1) ++(270+30:1) circle[radius=2pt];
            \fill (270+80/2:1) ++(270-30:1) circle[radius=2pt];
            
            \draw (0,0) -- (270-80/2:1);
            \fill (270-80/2:1) circle[radius=2pt];
            \draw (270-80/2:1) -- ++(270-30:1);
            \draw (270-80/2:1) -- ++(270+30:1);
            \fill (270-80/2:1) ++(270+30:1) circle[radius=2pt];
            \fill (270-80/2:1) ++(270-30:1) circle[radius=2pt];
            
            \fill (0,0) circle[radius=2pt];
            
            \node[left] at (270-80/2:1) {$[k-\alpha_1]$};
            \node[left] at ($(270-80/2:1)+(270-30:1)$) {$[\alpha_1]$};
            \node[left] at ($(270-80/2:1)+(270+30:1)$) {$[\alpha_1]$};
            
            \node[right] at (270+80/2:1) {$[k-\alpha_2]$};
            \node[right] at ($(270+80/2:1)+(270-30:1)$) {$[\alpha_2]$};
            \node[right] at ($(270+80/2:1)+(270+30:1)$) {$[\alpha_2]$};
            
            \node[right] at (0,0) {\,$[-\alpha_1] = [-\alpha_2]$};
        \end{tikzpicture}
        \caption{Merging of Leaves in $T_{2,d}$}
        \label{fig: mergeing of leaves in perfect binary}
    \end{figure}

    As we continue to move up the tree one level at a time, the same argument and induction apply to show that the parity of the label at the point of merger of any two lower branches is shifted by $k$ from the parities of the labels of its children, which again forces backwards parity compatibility.
\end{proof}

In addition, we need the following rather technical result in preparation of Lemma \ref{lem: perfect binary tree same row labels}.

\begin{lemma}\label{lem: perfect binary tree same row labels2}
    Let $k\in \Z$ and $n,d\in\Z^+$ with $n$ being even. Let $v_0$ denote the root of $T_{2,d}$, and let $v_1$ and $v'_1$ denote the children of $v_0$. Further, let $T$ and $T'$ denote the two rooted perfect binary subtrees of $T_{2,d}$ of height $d-1$ with roots $v_1$ and $v'_1$, respectively. Assume that $T_{2,d}$ admits a closed coloring with remainder $k\mod n$ that is constant on each of the respective levels of the subtrees $T$ and $T'$. Then $T_{2,d}$ also admits a closed coloring with remainder $k\mod n$ that is constant on every level of $T_{2,d}$.
\end{lemma}

\begin{proof}
    Given any closed coloring $\ell$ of $T_{2,d}$ with remainder $k\mod n$ such that $\ell$ is constant on each of the respective levels of the subtrees $T$ and $T'$, pick a graph automorphism $\sigma$ of $T_{2,d}$ of order two with $\sigma(v_1) =v'_1$. Note that $\sigma$ provides a graph isomorphism between $T$ and $T'$, and consider the new labeling $\ell'$ given by
    \[ \ell' = \frac{\ell + \ell\circ\sigma}{2}\,.\]
    With Lemma \ref{lem: perfect binary tree even same parity labels}, $\ell'$ is a $\Z$-labeling of $T_{2,d}$. Moreover, $\ell'$ is readily a closed coloring of $T_{2,d}$ with remainder $k\mod \frac n2$ that is constant on every level of $T_{2,d}$. However, while the condition
    \begin{align}
   \sum_{u\in N[v]} \ell'(u) \equiv k\mod n \label{eq6.3}
    \end{align}
    is satisfied for $v := v_0$, in general, $\sum_{u\in N[v]} \ell'(u) \equiv (k+ \frac n2) \mod n$ may also be possible. In the following, we will demonstrate that one can adjust the coloring $\ell$ on the vertices of the subtree $T$ by increasing some of the labels by $n$ such that Condition \eqref{eq6.3} holds for all vertices $v$ of~$T_{2,d}$.
    
    For this, for each $i=2,\ldots,d$, pick a vertex $v_i$ of $T$ at distance $i$ from $v_0$. We will now adjust
    the coloring $\ell$ on $T$ recursively, starting at the root of $T$ and moving down level by level to the bottom of the tree $T$.

    Starting with the root $v_1$ of $T$, if Condition \eqref{eq6.3} is satisfied for $v := v_1$, we will keep the current label of $v_1$. However, if Condition~\eqref{eq6.3} fails for $v := v_1$, then $\sum_{u\in N[v_1]} \ell'(u) \equiv (k+ \frac n2) \mod n$ and we increase the label of $v_1$ by $n$. In either case, one verifies
    that we found a closed coloring $\ell$ such that Condition \eqref{eq6.3} is satisfied for both $v := v_0$ and $v := v_1$.

    Continue with $v_2$. If Condition \eqref{eq6.3} is satisfied for $v := v_2$, we will keep the current labels of the children of $v_1$. However, if Condition~\eqref{eq6.3} fails for $v := v_2$, then $\sum_{u\in N[v_2]} \ell'(u) \equiv (k+ \frac n2) \mod n$ and we increase the labels of the children of $v_1$ by $n$. In either case, one verifies that we found a closed coloring $\ell$ such that Condition \eqref{eq6.3} is satisfied for all $v \in \{v_0, v_1, v_2\}$.

    As we continue to move down the tree $T$ one level at a time, we will finally end up with a closed coloring $\ell$ of $T_{2,d}$ with remainder $k\mod n$ that is constant on each of the respective levels of the subtrees $T$ and $T'$ such that Condition \eqref{eq6.3} is satisfied for all $v \in \{v_0, v_1,\ldots, v_d\}$. Thus, the corresponding $\ell'$ will be a closed coloring with remainder $k\mod n$.
\end{proof}

Our next result tells us that it suffices to limit ourselves to closed colorings which are constant on every level of the tree $T_{2,d}$.

\begin{lemma}\label{lem: perfect binary tree same row labels}
    Let $k\in \Z$ and $n,d\in\Z^+$. If $T_{2,d}$ admits a closed coloring with remainder $k\mod n$, then it admits one in which nodes within each level of the tree $T_{2,d}$ share the same label.
\end{lemma}

\begin{proof}
    We first consider the case of odd $n$. Let $\sigma$ denote any graph automorphism of $T_{2,d}$ of order two and let $\ell$ be a closed coloring of $T_{2,d}$ with remainder $k\mod n$. Note that $\sigma$ preserves levels, and consider the new labeling $\ell'$ given by
    \[ \ell' = \frac{\ell + \ell\circ\sigma}{2}\,,\]
    where division by $2$ is interpreted as multiplication by a multiplicative inverse of $2 \mod n$. 
    
    Then $\ell'$ is still a closed coloring of $T_{2,d}$ with remainder $k\mod n$, but it is constant on the orbits of $\sigma$. Starting at the bottom of $T_{2,d}$ and working upwards with the automorphisms that flip the branches below a vertex, eventually gives the result.

    We next consider the case of even $n$. Again, let $\ell$ be a closed coloring of $T_{2,d}$ with remainder $k\mod n$. This time, starting at the bottom of $T_{2,d}$ and working upwards with the help of Lemma \ref{lem: perfect binary tree same row labels2} gives the desired result.
\end{proof}

\begin{theorem}\label{thm: perfect binary trees}
    Let $k \in \Z$ and $n,d\in \Z^+$. Let 
    \[ f(\alpha, t) 
    := \frac{(k-\alpha)t + \alpha}{(1-t)(2t^2 + t +1)} \] and expand $f$ as a power series in $t$ as
    \[ f(\alpha,t) = \sum_{i=0}^\infty f_i(\alpha) t^i.\]
    Then $\chi_{n,k}(T_{2,d})$ exists if and only if there exists $\alpha\in\Z$ such that
    \[ f_{d+1}(\alpha) \equiv 0 \mod n.\]
\end{theorem}

\begin{proof}
    By Lemma \ref{lem: perfect binary tree same row labels}, a closed coloring with remainder $k\mod n$ of $T_{2,d}$ exists if and only if one exists with constant labels within each level of the tree $T_{2,d}$. Suppose we label the nodes of each level of $T_{2,d}$, starting at the bottom and ending at the root of the tree, with $x_0, x_1, \ldots, x_d$, respectively. This provides a closed coloring with remainder $k\mod n$ if and only if
    \begin{enumerate}
        \item $x_0 + x_1 \equiv k \mod n$,
        \item $x_i + x_{i-1} + 2x_{i-2} \equiv k \mod n$ \,\, for $2\leq i \leq d$,
        \item $x_d + 2x_{d-1} \equiv k \mod n$.
    \end{enumerate}

    Now Equations $(1)$ and $(2)$ determine all $x_i$ in terms of $x_0$. Equation~$(3)$ then determines if the resulting labeling ends up being a closed coloring.

    For $\alpha\in\Z$, use $(1)$ and $(2)$ to recursively define
    \[ x_0 = \alpha, \, x_1 = k - \alpha, \, x_i = k - x_{i-1} - 2x_{i-2}\quad \text{for $i \geq 2$.}\]
    Note that we have a closed coloring if and only if Equation $(3)$ is satisfied if and only if $x_{d+1} \equiv 0 \mod n$ for some $\alpha \in \Z$.  
    
    Define the formal power series
    \[ f(\alpha, t) = \sum_{i=0}^\infty x_i \, t^i.\]
    Using the recursive definition of $x_i$ shows that 
    \[ (1+t+2t^2)f(\alpha, t) = \alpha + \sum_{i=1}^\infty kt^i= \alpha + \frac{kt}{1-t}\,.\]
    From this, it follows that 
    \[ f(\alpha, t) = \frac{(k-\alpha)t + \alpha}{(1-t)(2t^2 + t +1)}\,.\qedhere \]
    %
    %
\end{proof}

One may calculate that 
\begin{align*}
    f(\alpha,t) =& \alpha + (k-\alpha)t -\alpha t^2 +(3\alpha-k)t^3
        +(2k-\alpha)t^4 +(k-5\alpha)t^5 \\
        & +(7\alpha-4k)t^6
        +3(\alpha+k)t^7 +(6k-17\alpha)t^8 +11(\alpha - k)t^9 \\
        & +23\alpha t^{10} +(23k-45\alpha)t^{11} -(\alpha+22k)t^{12}
        +(91\alpha-23k)t^{13} \\
        & +(68k-89\alpha)t^{14}
        -3(31\alpha+7k)t^{15} +(271\alpha-114k)t^{16} +\ldots\,.
\end{align*}
From this, we may read off that a closed coloring with remainder $k\mod n$ of $T_{2,d}$ exists for all choices of $k\in \Z$ and $n\in\Z^+$ when $d =0, 1, 3, 6, 8, 9, 11$. 
However, $d=2$ requires $(3,n)\mid k$, $d=4$ requires $(5,n)\mid k$, $d=5$ requires $(7,n)\mid k$, $d=7$ requires $(17,n)\mid k$, and $d=10$ requires $(45,n)\mid k$. The reader may read off the additional requirements up to $d=15$ from the expansion above.
It would be interesting to see if patterns could be discerned from the power series.

\section{Generalized Petersen Graphs}\label{sec: gen petersen}

Write $G(m,j)$ for the \emph{generalized Petersen graph} where $m,j\in\Z^+$ with $m\geq 3$ and $1\leq j < \frac{m}{2}$. We will use the notation $V=\{v_i, u_i \mid i\in \Z_m\}$ for the vertex set of $G(m,j) =(V,E)$ with corresponding edge set
\[ E= \{v_i v_{i+[1]}, v_i u_i, u_i u_{i+[j]} \mid i\in \Z_m\}, \]
where $[1], [j] \in \Z_m$ denote congruence classes modulo $m$. 
We may refer to the $v_i$ as the \emph{exterior vertices} and the $u_i$ as the \emph{interior vertices}.
Observe that the interior vertices break up into $(m,j)$ cycles of size~$\frac{m}{(m,j)}$.


    


As $G(m,j)$ is $3$-regular, the constant labeling of $1$ generates a closed coloring with remainder $4\mod n$ for any $n\in \Z^+$. Because the sum of closed colorings with remainders $k_i\mod n$, $i=1,2$, is a closed coloring with remainder $(k_1+k_2)\mod n$, it follows that the existence of $\chi_{n,k}(G(m,j))$ depends only on the residue class of $k\mod 4$. In particular, a closed coloring with remainder $k\mod n$ always exists when $k\equiv 0\mod 4$. 

Moreover, as the product of a constant $c\in\Z$ with a closed coloring with remainder $k\mod n$ is a closed coloring with remainder $ck\mod n$, it follows that $\chi_{n,1}(G(m,j))$ exists if and only if $\chi_{n,-1}(G(m,j))$ exists. Furthermore, if $\chi_{n,1}(G(m,j))$ exists, then $\chi_{n,k}(G(m,j))$ exists for all $k$.

In summary, the analysis for the existence of $\chi_{n,k}(G(m,j))$ is reduced to the study of $k=1$ (which gives existence of all $\chi_{n,k}(G(m,j))$) and, when $\chi_{n,1}(G(m,j))$ does not exist, to the study of $k=2$. If both of these fail, $\chi_{n,k}(G(m,j))$ exists if and only if $4\mid k$.

We begin with the following result for $k=1$.
\begin{theorem}\label{lem: k=1 main lem for petersen n,k's}
    Let $n,m,j\in \Z^+$. 
    \begin{enumerate}
        \item If $4\nmid n$, then $\chi_{n,1}(G(m,j))$ exists.
        \item If $4 \mid n$ and $2\nmid m$, then $\chi_{n,1}(G(m,j))$ does not exist.
        \item If $4 \mid n$, $2\mid m$, and $2\nmid j$, then $\chi_{n,1}(G(m,j))$ exists if and only if $4\mid m$.
        \item If $4\mid n$, $8 \nmid n$, $2\mid m$, and $2\mid j$, then $\chi_{n,1}(G(m,j))$ exists.
        \item If $8\mid n$, $2\mid m$, $4 \nmid m$, and $2\mid j$, then $\chi_{n,1}(G(m,j))$ does not exist.
        \item If $16\mid n$, $4\mid m$, $8 \nmid m$, and $2\mid j$, then $\chi_{n,1}(G(m,j))$ does not exist.
        \item If $8\mid n$, $16\nmid n$, $4\mid m$, and $2\mid j$, existence of $\chi_{n,1}(G(m,j))$ is not currently known.
        \item If $8\mid n$, $8\mid m$, and $2\mid j$, existence of $\chi_{n,1}(G(m,j))$ is not currently known.
    \end{enumerate}
\end{theorem}

The proof of Theorem \ref{lem: k=1 main lem for petersen n,k's} will follow from Lemmas \ref{lem: k=1 and (4,n)=1,2 for petersen n,k's}, \ref{lem: k=1 and (4,n)=4 nothing for m odd, petersen n,k's}, \ref{lem: k=1, (4,n)=4, m even, j odd, petersen nk's}, and \ref{lem: k=1, (4,n)=4, m even, 4 not divide m, petersen nk's} below.
Figure \ref{fig: k=1 Petersen pic} gives a visual representation of the existence and nonexistence of $\chi_{n,1}(G(m,j))$ from Theorem \ref{lem: k=1 main lem for petersen n,k's}.

    \begin{figure}[H]
        \centering
        \begin{tikzpicture}[row sep=1cm, column sep=1cm]
          
          \fill[lightgray] (2.5, .5) rectangle (5, 4.5);
          \fill[lightgray] (5, .5) rectangle (7.5, 2.5);
          \fill[lightgray] (5, 2.5) -- (7.5, 2.5) -- (7.5, 3.5) -- cycle;
          \fill[lightgray] (5, 2.5) -- (7.5, 2.5) -- (7.5, 3.5) -- cycle;
          \fill[lightgray] (5, 2.5) -- (7.5, 2.5) -- (7.5, 3.5) -- cycle;
          \fill[lightgray] (7.5, .5) -- (7.5, 1.5) -- (10, 1.5) -- cycle;
          \fill[lightgray] (7.5, 1.5) -- (7.5, 2.5) -- (10, 2.5) -- cycle;
          
          \fill[pattern=dots] (7.5, 2.5) rectangle (10, 3.5);
          \fill[pattern=dots] (5, 3.5) rectangle (10, 4.5);
          \fill[pattern=dots] (5, 2.5) -- (5, 3.5) -- (7.5, 3.5) -- cycle;
          \fill[pattern=dots] (8.9, 1.5) -- (8.9, 2) -- (10, 2.5) -- (10, 1.5) -- cycle;

          \draw (-.5, .5) -- (-.5, 4.5);
          
          \draw (2.5, .5) -- (2.5, 5.5);
          \draw (5, .5) -- (5, 5.5);
          \draw (7.5, .5) -- (7.5, 5.5);
          \draw (10, .5) -- (10, 5.5);

          \draw (-.5, .5) -- (10, .5);
          \draw (-.5, 1.5) -- (10, 1.5);
          \draw (-.5, 2.5) -- (10, 2.5);
          \draw (-.5, 3.5) -- (10, 3.5);
          \draw (-.5, 4.5) -- (10, 4.5);

          \draw (2.5, 5.5) -- (10, 5.5);

          \draw (5, 2.5) -- (7.5, 3.5); 
          \draw (7.5, 1.5) -- (10, 2.5); 
          \draw (7.5, .5) -- (10, 1.5);
          
          \node at (1,4) {$2\nmid m$};
          \node at (1,3) {$2\mid m$ but $4\nmid m$};
          \node at (1,2) {$4\mid m$ but $8\nmid m$};
          \node at (1,1) {$8\mid m$};
          
          \node at (3.75,5) {$4\nmid n$};
          \node at (6.25,5) {$4\mid n$ but $8\nmid n$};
          \node at (8.75,5) {$8\mid n$};
          \node at (9.45, 1.85) {$16 \mid n$};
          
          \node at (9.5, .85) {?};
          \node at (8.65, 1.72) {?};
          
        \end{tikzpicture}
        \caption{Diagonals: $2\nmid j$ Northwest, $2 \mid j$ Southeast\\
                Shaded Regions: $\chi_{n,1}(G(m,j))$ exists. \\
                Dotted Regions: $\chi_{n,1}(G(m,j))$ does not exist.}
        \label{fig: k=1 Petersen pic}
    \end{figure}

\begin{lemma}\label{lem: k=1 and (4,n)=1,2 for petersen n,k's}
    Let $n,m,j\in \Z^+$. 
    If $4\nmid n$, then there exists a closed coloring of $G(m,j)$ with remainder $1\mod n$.
\end{lemma}

\begin{proof}
    If $(4,n)=1$, it is possible to solve the equation $4\alpha\equiv 1\mod n$ for some $\alpha\in\Z$. In that case, the constant labeling of $\alpha$ gives rise to the existence of $\chi_{n,1}(G(m,j))$, and $\chi_{n,1}(G(m,j))=\chi(G(m,j))$.
    
    We turn to the case of $(4,n)=2$, where we will demonstrate the existence of a closed coloring that is constant on the exterior vertices and constant on the interior vertices. For this, we must be able to solve the equations
    \[3\alpha + \beta \equiv 1 \mod n, \]
    \[\alpha + 3\beta \equiv 1 \mod n\ \]
    for some $\alpha, \beta \in \Z$.

    It is straightforward to see that these equations require that $2\alpha \equiv 2\beta \mod n$. In fact, we will take
    \[ \beta = \alpha + \frac{n}{2}\,.\]
    With this ansatz, solving the desired equations is equivalent to solving
    \[ 4\alpha\equiv \left( 1 + \frac{n}{2} \right) \mod n.\]
    In turn, this has a solution if and only if $(4,n)\mid(1+\frac{n}{2})$. However, as $(4,n)=2$ and $\frac{n}{2}$ is odd, we are done.
\end{proof}

\begin{lemma}\label{lem: k=1 and (4,n)=4 nothing for m odd, petersen n,k's}
    Let $n,m,j\in \Z^+$ and suppose $4\mid n$. If $\chi_{n,1}(G(m,j))$ exists, then $m$ is even.
\end{lemma}

\begin{proof}
     Suppose $\ell$ is a closed coloring with remainder $1\mod n$. Define
     \[ V = \sum_{i\in\Z_m} \ell(v_i) \text{ and } 
        U = \sum_{i\in\Z_m} \ell(u_i). \]
    Then 
     \[ m = \sum_{i\in\Z_m}1 \equiv 
        \sum_{i\in\Z_m} \sum_{u\in N[v_i]}\ell(u) = (3V+U) \mod n \quad \mbox{ and }  \] 
     \[ m= \sum_{i\in\Z_m}1 \equiv 
        \sum_{i\in\Z_m} \sum_{u\in N[u_i]}\ell(u) = (V+3U) \mod n. \qquad\ \ \ \mbox{} \]
     This implies 
     $2m \equiv (4V+4U) \mod n$,
    which forces $m$ to be even.
\end{proof}

\begin{lemma}\label{lem: k=1, (4,n)=4, m even, j odd, petersen nk's}
    Let $n,m,j\in \Z^+$ and suppose $4\mid n$, $2\mid m$, and $2\nmid j$. Then $\chi_{n,1}(G(m,j))$ exists if and only if $4\mid m$. 
\end{lemma}

\begin{proof}
    Suppose first that $\chi_{n,1}(G(m,j))$ exists. 
    Proceed with a refinement of the proof of Lemma \ref{lem: k=1 and (4,n)=4 nothing for m odd, petersen n,k's} in which $V$ and $U$ are broken into their even and odd parts. For $\pi\in\{[0],[1]\}\subseteq \Z_m$, viewed rather as an element of $\Z_2$ whenever used in superscripts, define
     \[ V^\pi=\sum_{i\in \, 2\Z_m +\pi} \ell(v_i) \text{ and } 
        U^\pi=\sum_{i\in \, 2\Z_m +\pi} \ell(u_i).  \]
     Then we get the equations, using $2\nmid j$ in the second set below,
     \[ \frac{m}{2} = \sum_{i\in \, 2\Z_m +\pi}1 \equiv 
        \sum_{i\in \, 2\Z_m +\pi} \sum_{\, u\in N[v_i]}\ell(u) = (V^\pi + 2V^{\pi+[1]} + U^\pi) \mod n \quad \mbox{and}  \] 
     \[ \frac{m}{2} = \sum_{i\in \, 2\Z_m +\pi}1 \equiv 
        \sum_{i\in \, 2\Z_m +\pi} \sum_{\, u\in N[u_i]}\ell(u) = (V^\pi + U^\pi + 2U^{\pi+[1]}) \mod n. \qquad\ \ \ \mbox{} \]

    Subtracting yields $2V^\pi \equiv 2U^\pi \mod n$ 
    while adding gives \[m\equiv 2V^\pi + 2V^{\pi+[1]}+ 2U^\pi+ 2U^{\pi+[1]} \equiv (4V^\pi + 4V^{\pi+[1]}) \mod n,\]
    which forces $4\mid m$.

    Now suppose $4\mid m$ and define a $\Z$-labeling $\ell$ of $G(m,j)$ by
    \[ \ell(v_i)= \begin{cases}
        1 & \text{if } 4\mid i,\\
        0 & \text{else},
    \end{cases} \text{  and  } \ell(u_i)=\begin{cases}
        1 & \text{if }  4\mid (i-[2]),\\ 
        0 & \text{else}.
    \end{cases}\]
    It is straightforward to verify that this gives a closed coloring with remainder $1\mod n$.
\end{proof}

\begin{lemma}\label{lem: k=1, (4,n)=4, m even, 4 not divide m, petersen nk's}
    Let $n,m,j\in \Z^+$ and suppose $4\mid n$, $2\mid m$, and $2\mid j$. Then existence of $\chi_{n,1}(G(m,j))$  is determined as follows:
    \begin{itemize}
        \item If $8 \nmid n$, then $\chi_{n,1}(G(m,j))$ exists.
        \item If $8 \mid n$ and $4 \nmid m$, then $\chi_{n,1}(G(m,j))$ does not exist.
        \item If $16 \mid n$, $4\mid m$, and $8 \nmid m$, then $\chi_{n,1}(G(m,j))$ does not exist.
        \item For $8 \mid n$, $16\nmid n$, and $4 \mid m$, the existence of $\chi_{n,1}(G(m,j))$ is not currently known.
        \item For $8 \mid n$ and $8 \mid m$, the existence of $\chi_{n,1}(G(m,j))$ is not currently known.
    \end{itemize}
\end{lemma}

\begin{proof}
    We use the notation of $V^\pi$ and $U^\pi$ from Lemma \ref{lem: k=1, (4,n)=4, m even, j odd, petersen nk's} above and suppose that $\chi_{n,1}(G(m,j))$ exists. Since $2\mid j$, we now get 
    \begin{align}
    \label{eq7.1} \frac{m}{2} =\!\! \sum_{i\in \, 2\Z_m +\pi}\!\! 1 \equiv \!\!
        \sum_{i\in \, 2\Z_m +\pi} \sum_{\, u\in N[v_i]}\!\! \ell(u) = (V^\pi + 2V^{\pi+[1]} + U^\pi) \mod n\\
    \nonumber  \mbox{ and } \,\, \frac{m}{2} = \sum_{i\in \, 2\Z_m +\pi}1 \equiv 
        \sum_{i\in \, 2\Z_m +\pi} \sum_{\, u\in N[u_i]}\ell(u) = (V^\pi + 3U^\pi) \mod n.
    \end{align}

    Subtracting gives $2V^{\pi+[1]}\equiv 2U^\pi \mod n$, so that 
    \[U^\pi \equiv \left(V^{\pi+[1]} +\delta_\pi\frac{n}{2}\right) \mod n\]
    for some $\delta_\pi \in \{0,1\}$.
    Substituting this back into our two initial equations, both equations reduce to 
    \begin{align}
    \frac{m}{2} \equiv \left(V^\pi + 3V^{\pi+[1]} +\delta_{\pi}\frac{n}{2}\right) \mod n \label{eq7.2}
    \end{align}
    for $\pi \in \{[0],[1]\}$.
    In particular,
    \[\left(V^{[0]} + 3V^{[1]} +\delta_{[0]}\frac{n}{2}\right) \equiv \left(V^{[1]} + 3V^{[0]} +\delta_{[1]}\frac{n}{2}\right) \mod n,\]
    hence $2V^{[1]} \equiv (2V^{[0]} +(\delta_{[1]}-\delta_{[0]})\frac{n}{2}) \mod n$.
    As a result, we must have
    $V^{[1]} \equiv (V^{[0]} +(\delta_{[1]}-\delta_{[0]})\frac{n}{4} + \delta\frac{n}{2}) \mod n$ for some $\delta\in\{0,1\}$.
    Substituting back into Equation \eqref{eq7.2}, we end up with the requirement
    \[ \frac{m}{2} \equiv \left(4V^{[0]} -(\delta_{[1]}+\delta_{[0]})\frac{n}{4} + \delta \frac{n}{2}\right) \mod n.\]
    
    In turn, this necessitates 
    \[4\mid \Bigl[\frac{m}{2} + (\delta_{[1]}+\delta_{[0]})\frac{n}{4} + \delta \frac{n}{2}\Bigr].\]
    For $\frac{n}{4}\equiv 0\mod 4$ this requires $\frac{m}{2}\equiv 0\mod 4$, and for $\frac{n}{4}\equiv 2\mod 4$ this requires either $\frac{m}{2}\equiv 0\mod 4$ or $\frac{m}{2}\equiv 2\mod 4$. 
    In summary, existence of a closed coloring with remainder $1\mod n$ fails in the following cases:
    \begin{itemize}
        \item $16\mid n$ and $8 \nmid m$,
        \item $8\mid n$, $16\nmid n$, and $4 \nmid m$.
    \end{itemize}

    To examine the existence of closed colorings with remainder $1\mod n$,
    look for one that 
    is constant on each of the sets $\{v_{2i+\pi} \mid i \in \Z_m\}$ and $\{u_{2i+\pi}\mid i \in \Z_m\}$. Write the labels as $a_\pi$ and $b_\pi$, respectively. Then a closed coloring with remainder 
    $1\mod n$ of this form exists if and only if
    \begin{align*}
    1 \equiv (a_\pi + 2a_{\pi+[1]} + b_\pi) \mod n & \quad \mbox{ and }\\
    1 \equiv (a_\pi + 3b_\pi) \mod n &
    \end{align*}
    can be solved, which is Equations \eqref{eq7.1} with $\frac{m}{2}, V^\pi, U^\pi$ replaced by $1, a_\pi, b_\pi$. As seen above, this can be done if and only if
    \[ 1 \equiv \left(4a_{[0]} -(\delta_{[1]}+\delta_{[0]})\frac{n}{4} + \delta \frac{n}{2}\right) \mod n\]
    for some $\delta_{[1]},\delta_{[0]},\delta \in \{0,1\}$, which can be achieved if and only if
    \[4\mid \Bigl[1 + (\delta_{[1]}+\delta_{[0]})\frac{n}{4} + \delta \frac{n}{2}\Bigr].\]
    In turn, this can be done if and only if $8\nmid n$.
    
    However, when $8\mid n$, the question of existence remains open. Though the above labeling scheme fails, more exotic labeling methods may be possible in some cases. This leaves us with the open cases $8\mid n$, $16\nmid n$, $4\mid m$ and $8\mid n$, $8\mid m$.
\end{proof}

Now we move on to the case $k=2$, which we only need to consider in the cases when $\chi_{n,k}(G(m,j))$ does not exist for $k=1$. By Theorem~\ref{lem: k=1 main lem for petersen n,k's}, this happens when 
\begin{enumerate}
    \item $4\mid n, 2\nmid m$,
    \item $4\mid n, 8\nmid n, 2\mid m, 4\nmid m, 2\nmid j$,
    \item $8\mid n, 2\mid m, 4\nmid m$,
    \item $16\mid n, 4\mid m, 8\nmid m, 2\mid j$,
    \item possible subcases of $8\mid n, 16\nmid n, 4\mid m, 8\nmid m, 2\mid j$,
    \item possible subcases of $8\mid n, 8\mid m, 2\mid j$.
\end{enumerate}
Our collected findings are as follows:

\begin{theorem}\label{thm: k=2 summary}
    Let $n,m,j\in\Z^+$.
    \begin{enumerate}
        \item If $8\nmid n$, then $\chi_{n,2}(G(m,j))$ exists.
        \item If $8\mid n$ and $2\nmid m$, then $\chi_{n,2}(G(m,j))$ does not exist.
        \item If $8\mid n$, $2\mid m$, and $2\nmid j$, then $\chi_{n,2}(G(m,j))$ exists.
        \item If $8\mid n$, $2\mid m$, $4\nmid m$, then $\chi_{n,2}(G(m,j))$ exists if and only if $16 \nmid n$.
        \item If $8\mid n$, $4\mid m$, and $2\mid j$, then existence of $\chi_{n,2}(G(m,j))$ is not currently known.
    \end{enumerate}
\end{theorem}

The proof of Theorem \ref{thm: k=2 summary} follows from Theorem \ref{lem: k=1 main lem for petersen n,k's} and Lemmas~\ref{lem: k=2 lem 1}, \ref{lem: k=2 lem 2}, and \ref{lem: k=2 lem 3}.
For $k=2$, Figure \ref{fig: k=2 Petersen pic} displays a visual representation for the existence of a closed coloring   of $G(m,j)$ with remainder $ k\mod n $.
    \begin{figure}[H]
        \centering
        \begin{tikzpicture}[row sep=1cm, column sep=1cm]
          
          \fill[lightgray] (2.5, .5) rectangle (5, 4.5);
          \fill[lightgray] (5, .5) rectangle (7.5, 2.5);
          \fill[lightgray] (5, 3.5) rectangle (7.5, 4.5);
          \fill[lightgray] (5, 2.5) -- (5, 3.5) -- (7.5, 3.5) -- cycle;
          \fill[lightgray] (5, 2.5) -- (7.5, 2.5) -- (7.5, 3.5) -- cycle;
          \fill[lightgray] (5, 2.5) -- (7.5, 2.5) -- (7.5, 3.5) -- cycle;
          \fill[lightgray] (5, 2.5) -- (7.5, 2.5) -- (7.5, 3.5) -- cycle;
          \fill[lightgray] (7.5, .5) -- (7.5, 1.5) -- (10, 1.5) -- cycle;
          \fill[lightgray] (7.5, 1.5) -- (7.5, 2.5) -- (10, 2.5) -- cycle;
          \fill[lightgray] (7.5, 2.5) -- (7.5, 3.5) -- (10, 3.5) -- cycle;
          \fill[lightgray] (7.5, 2.5) -- (8.75, 2.5) -- (8.75, 3) -- cycle;
          
          \fill[pattern=dots] (7.5, 3.5) rectangle (10, 4.5);
          \fill[pattern=dots] (8.75, 2.5) -- (8.75, 3) -- (10, 3.5) -- (10, 2.5) -- cycle;

          \draw (-.5, .5) -- (-.5, 4.5);
          
          \draw (2.5, .5) -- (2.5, 5.5);
          \draw (5, .5) -- (5, 5.5);
          \draw (7.5, .5) -- (7.5, 5.5);
          \draw (10, .5) -- (10, 5.5);

          \draw (-.5, .5) -- (10, .5);
          \draw (-.5, 1.5) -- (10, 1.5);
          \draw (-.5, 2.5) -- (10, 2.5);
          \draw (-.5, 3.5) -- (10, 3.5);
          \draw (-.5, 4.5) -- (10, 4.5);

          \draw (2.5, 5.5) -- (10, 5.5);

          \draw (7.5, 2.5) -- (10, 3.5); 
          \draw (7.5, 1.5) -- (10, 2.5); 
          \draw (7.5, .5) -- (10, 1.5);
          
          \node at (1,4) {$2\nmid m$};
          \node at (1,3) {$2\mid m$ but $4\nmid m$};
          \node at (1,2) {$4\mid m$ but $8\nmid m$};
          \node at (1,1) {$8\mid m$};
          
          \node at (3.75,5) {$4\nmid n$};
          \node at (6.25,5) {$4\mid n$ but $8\nmid n$};
          \node at (8.75,5) {$8\mid n$};
          
          \node at (9.5, .85) {?};
          \node at (9.5, 1.85) {?};
          \node at (9.4, 2.85) {$16 \mid n$};
          
        \end{tikzpicture}
        \caption{Diagonals: $2\nmid j$ Northwest, $2 \mid j$ Southeast\\
                Shaded Regions: $\chi_{n,2}(G(m,j))$ exists. \\
                Dotted Regions: $\chi_{n,2}(G(m,j))$ does not exist.}
        \label{fig: k=2 Petersen pic}
    \end{figure}

\begin{lemma}\label{lem: k=2 lem 1}
    Let $n,m,j\in\Z^+$. Suppose $4\mid n$ and $2\nmid m$. Then a closed coloring of $G(m,j)$ with remainder $2 \mod n$ exists if and only if $8\nmid n$.
\end{lemma}
\begin{proof}
    We follow notation and ideas similar to Lemma \ref{lem: k=1 and (4,n)=4 nothing for m odd, petersen n,k's} with the exception that $2m = \sum_{i\in\Z_m}2$ replaces $m = \sum_{i\in\Z_m}1$. Hence the new equations become 
    \begin{align}
    2m \equiv (3V + U) \mod n \quad \mbox{ and }\label{eq7.3}\\
    2m \equiv (V+3U) \mod n.  \qquad\ \ \ \mbox{}\nonumber 
    \end{align}
    
    By methods similar to Lemma \ref{lem: k=1, (4,n)=4, m even, 4 not divide m, petersen nk's}, we see that $2V \equiv 2U \mod n$, hence $V\equiv (U+\delta\frac{n}{2}) \mod n$ for some $\delta \in\{0,1\}$, and our Equations~\eqref{eq7.3} reduce to one single equation
    \[2m\equiv \left( 4U +\delta\frac{n}{2}\right) \mod n.\]
        This equation has a solution if and only if  
    \[2\mid \left(m + \delta\frac{n}{4} \right).\] 
    As $m$ is odd, this forces $\delta = 1$ and $\frac{n}{4}$ to be odd. Thus, $8\nmid n$ is necessary for the existence of a closed coloring with remainder $2 \mod n$.

    If $8\nmid n$, a closed coloring of $G(m,j)$ with remainder $2 \mod n$ may be obtained by the techniques found in   
    Lemma \ref{lem: k=1 and (4,n)=1,2 for petersen n,k's} using a
    labeling that is constant on the exterior vertices and constant on the interior vertices. For this, we must be able to solve the equations
    \[3\alpha + \beta \equiv 2 \mod n, \]
    \[\alpha + 3\beta \equiv 2 \mod n\ \]
    for some $\alpha, \beta \in \Z$.

    This requires $2\alpha \equiv 2\beta \mod n$ and, in fact, we will take
    $\beta = \alpha + \frac{n}{2}.$
    With this ansatz, solving the desired equations is equivalent to solving
    \[ 4\alpha\equiv \left( 2 + \frac{n}{2} \right) \mod n.\]
    In turn, this has a solution if and only if $2\mid (1+\frac{n}{4})$. As $\frac{n}{4}$ is odd, we are done.
\end{proof}

\begin{lemma}\label{lem: k=2 lem 2}
    Let $n,m,j\in\Z^+$. Suppose $2\mid m$ and $2\nmid j$. Then there exists a closed coloring  of $G(m,j)$ with remainder $2 \mod n$. 
\end{lemma}
\begin{proof}
    Define a $\Z$-labeling $\ell$ of $G(m,j)$ by
    \[ \ell(v_i) = \ell(u_i) = \begin{cases}
        1 & \text{if } 2\mid i,\\
        0 & \text{else}.
    \end{cases} \]
    It is straightforward to verify that this gives a closed coloring with remainder $2\mod n$.
\end{proof}

\begin{lemma}\label{lem: k=2 lem 3}
     Let $n,m,j\in\Z^+$. Suppose $8\mid n$, $2\mid m$, $4\nmid m$, and $2\mid j$. Then there exists a closed coloring of $G(m,j)$ with remainder $2 \mod n$ if and only if $16\nmid n$.
\end{lemma}
\begin{proof}
    We follow the ideas and notation from Lemma \ref{lem: k=1, (4,n)=4, m even, 4 not divide m, petersen nk's}. 
    First, suppose $\chi_{n,2}(G(m,j))$ exists. Since $2\mid j$, we get the system of equations
    \begin{align*}
     m \equiv 
         (V^\pi + 2V^{\pi+[1]} + U^\pi) \mod n\ & \quad \mbox{ and }\\
    m \equiv 
         (V^\pi + 3U^\pi) \mod n, &
    \end{align*}
    which is Equations \eqref{eq7.1} with $\frac m2$ replaced by $m$.
    As seen in the proof of Lemma \ref{lem: k=1, (4,n)=4, m even, 4 not divide m, petersen nk's}, this necessitates $4\mid [m + (\delta_{[1]}+\delta_{[0]})\frac{n}{4} + \delta \frac{n}{2}]$, hence 
    \[2\mid \left[\frac m2 + (\delta_{[1]}+\delta_{[0]})\frac{n}{8} \right].\]
    In particular, for $\frac{m}{2} \equiv 1 \mod 2$, we must have $\frac{n}{8}\equiv 1 \mod 2$ as well.  

    To examine the existence of closed colorings with remainder $2\mod n$,
    look for one that 
    is constant on each of the sets $\{v_{2i+\pi} \mid i \in \Z_m\}$ and $\{u_{2i+\pi}\mid i \in \Z_m\}$. Write the labels as $a_\pi$ and $b_\pi$, respectively. Then a closed coloring with remainder 
    $2\mod n$ of this form exists if and only if
    \begin{align*}
    2 \equiv (a_\pi + 2a_{\pi+[1]} + b_\pi) \mod n & \quad \mbox{ and }\\
    2 \equiv (a_\pi + 3b_\pi) \mod n &
    \end{align*}
    can be solved, which is Equations \eqref{eq7.1} with $\frac m2, V^\pi, U^\pi$ replaced by $2, a_\pi, b_\pi$. As seen in the proof of Lemma \ref{lem: k=1, (4,n)=4, m even, 4 not divide m, petersen nk's}, this can be done if and only if $4\mid [2 + (\delta_{[1]}+\delta_{[0]})\frac{n}{4} + \delta \frac{n}{2}]$ for some $\delta_{[1]},\delta_{[0]},\delta \in \{0,1\}$, which simplifies to
    \[2\mid \left[1 + (\delta_{[1]}+\delta_{[0]})\frac{n}{8}\right].\]
    In turn, this happens if and only if $16\nmid n$. 
\end{proof}

\section{Concluding Remarks}

There still remain a few cases of high divisibility by $2$ where existence of a closed coloring with remainder $k\mod n$ for generalized Petersen graphs is undetermined, see Figures \ref{fig: k=1 Petersen pic} and \ref{fig: k=2 Petersen pic}. After this, determining the exact value of $\chi_{n,k}(G(m,j))$ would be desirable. 


\bibliographystyle{abbrvnat}
\bibliography{refs}

\end{document}